\documentclass[12pt,titlepage,aps,amssymb]{amsart}
\usepackage{graphicx,amsmath,amsthm, amssymb}
\newcommand{\ben}{\begin{enumerate}}
\newcommand{\een}{\end{enumerate}}
\newcommand{\tf}{\tilde f}

\newcommand{\T}{ \mathrm{T}}
\def\wh{\widehat}

\def \bc{\vskip 0cm \begin{center}}
\def \ec{\end{center} \vskip 0cm}
\def \be{\begin{equation}}
\def \ee{\end{equation}}
\def \bea{\begin{eqnarray}}
\def \eea{\end{eqnarray}}

\newtheorem{thm}{Theorem}[section]
\newtheorem{cor}[thm]{Corollary}
\newtheorem{lem}[thm]{Lemma}
\newtheorem{prop}[thm]{Proposition}

\theoremstyle{remark}
\newtheorem*{rem}{Remark}
\newtheorem{Def}[thm]{Definition}
\newtheorem{exm}{Example}

\begin{document}
\pagestyle{plain}
\begin{center}
\vspace{4cm}

{\Large {\bf Dynamics of shear homeomorphisms of tori and the Bestvina-Handel algorithm
}}\vskip 1cm
\large{Tali Pinsky and Bronislaw Wajnryb}
\end{center}
\vskip 2cm

\section*{Abstract}
Sharkovskii proved that the existence of a periodic orbit of period which is not a power of 2 in a one-dimensional dynamical system implies existence of infinitely many periodic orbits. We obtain an analog of  Sharkovskii's theorem for periodic orbits of shear homeomorphisms of the torus. This is done by obtaining a dynamical order relation on the set of simple orbits and simple pairs.
We then use this order relation for a global analysis of a
quantum chaotic physical system called the kicked
accelerated particle.

\section{Introduction}
Given a dynamical system $(X,f)$, a key question is which periodic orbits exist for this system. Since periodic orbits are in general difficult to compute, we would like to have the means to deduce their existence without having to actually compute them. 

Sharkovskii addressed the dynamics of continuous maps on the real line. He defined an order $\lhd$ on the natural numbers, Sharkovskii's order (see \cite{Misiurewicz}), and proved that the existence of a periodic orbit of a certain period $p$ implies the existence of an orbit of any period $q\lhd p$. We say the $q$ orbit is \emph{forced} by the $p$ orbit. This offers the means of showing the existence of many orbits if one can find a single orbit of ``large" period. For a dynamical system depending on a single parameter, if periodic orbits appear when we change the parameter, they must appear according to the Sharkovskii's order.  Hence, Sharkovskii's theorem gives the global structure of the appearance of periodic orbits for one dimensional systems. 
Ever since the eighties there has been interest in obtaining analogs for Sharkovskii's theorem for two dimensional systems (see \cite{Boy3} and \cite{Matzuoka}). 

A homeomorphism of a torus is said here to be of {\it shear type} if it is isotopic to one Dehn twist along a single closed curve. 
Let $h$ be a shear homeomorphism, and let $x$ be a periodic orbit of $h$. We
can then define the \emph{rotation number} of $x$, see discussion in Section \ref{simple section}. Thus, a rational number in the unit interval $[0,1)$ is associated to each orbit. 

We consider orbits up to conjugation: orbits $(x,f)$  and  $(y,g)$ are
\emph{similar} (of the same type) if there exists a homeomorphism $h$ of the torus $T^2$ such that $h$ is isotopic to the identity, $h$ takes
orbit $x$ onto orbit $y$  and   $hfh^{-1}$ is isotopic to $g$ rel $y$. 
We define below a specific family of periodic orbits we call simple orbits. In this family there is a unique element up to similarity corresponding to each rotation number; hence they can be specified by their rotation numbers. We emphasize it is not true in general that an orbit of a shear homeomorphism is characterized by its rotation number.

Simple orbits are analyzed in Section \ref{simple section}. As it turns out (see Remark in section 2), one simple orbit is indeed simple and does not force the existence of any other orbit. More generally a periodic orbit is of {\em twist type} if it does not force the existence of any orbit of different type with the same rotation number. It is tempting to conjecture that the simple orbits are the only orbits of twist type, but Lemma 2.4 shows that this is false. We give there an example of an orbit of twist type which is not simple. This example also shows that a periodic orbit with a given rotation number does not necessarily force a simple orbit of the same rotation number. 

 We turn in Section \ref{pairs} to analyze pairs of orbits. Two coexisting simple periodic orbits can form a simple pair and these are considered. The pairs do force some more interesting dynamics, as follows. We denote the integers by letters $p,q$ and the rational numbers by $r,s,t$ possibly with indices. For a pair of simple orbits of rotation numbers $\frac{q_1}{p_1}$ and $\frac{q_2}{p_2}$ to constitute a simple pair, it is necessary that the rotation numbers be Farey neighbors, i.e. $|p_2q_1-p_1q_2|=1$. We denote such a pair of rational numbers by $\frac{q_1}{p_1}\vee\frac{q_2}{p_2}$. 

We now define an order relation on the following set $\mathcal{P}$ of rational numbers and pairs in the unit interval,
\begin{eqnarray*}
\mathcal{P}=\{r| r\in\mathbb{Q}\cap[0,1)\}\cup\{r\vee s|r,s\in\mathbb{Q}\cap[0,1)\}.
\end{eqnarray*}
Define the order relation on $\mathcal{P}$ to be 
\begin{eqnarray*}
r\vee s \succcurlyeq t
\Leftrightarrow
t\in [
r,s]
\end{eqnarray*}
and
\begin{eqnarray*} 
r_1\vee r_2 \succcurlyeq s_1\vee s_2
\Leftrightarrow
s_1,s_2\in [r_1,r_2] .
\end{eqnarray*}
where we denote by $[r_1,r_2]$ the interval between $r_1$ and $r_2$, regardless of their order.
\begin{thm} \label{mainthm} The order relation $\succcurlyeq$ on $\mathcal{P}$ describes the dynamical
forcing of simple periodic orbits. Namely, the existence of a simple pair of periodic orbits with rotation numbers $r\vee s$ in $\mathcal{P}$ implies the
existence of all simple orbits and simple orbit pairs of rotation numbers smaller than
$r\vee s$
according to this order relation .
\end{thm}

This is the main result of this paper and the proof is completed
in Section \ref{order}. The idea of the proof is as follows. Consider the torus punctured on one or more periodic orbits of a homeomorphism $h$. Then $h$ induces an action on this punctured torus, and on its (free) fundamental  group. Now apply the Bestvina-Handel algorithm to this dynamical system. The idea of using the Bestvina-Handel algorithm was used by Boyland in \cite{Boy1} and he describes the general approach in \cite{Boy2}. In our case, after puncturing out a simple pair of orbits, applying the algorithm yields an isotopic homeomorphism which is pseudo-Anosov. The algorithm also offers a Markov partition for this system and we use the resulting symbolic representation to find that there are periodic orbits of each rotation number between the pair of numbers we started with. Then we directly analyze the structure of the pseudo-Anosov representative to show that all these orbits are in fact simple orbits. Furthermore any two of them corresponding to rotation numbers which are Farey neighbors form a simple pair. Finally we establish the isotopy stability of these orbits using results of Asimov and Franks \cite{AF} and Hall \cite{Hall}. Thus, the orbits exist for any homeomorphism for which the simple pair exists, and are forced by it.

One should compare this result with a very strong theorem of Doeff (see Theorem \ref{Doeff's theorem}), where existence of two periodic orbits of different periods for a given shear homeomorphism $h$  implies existence of periodic orbit of every intermediate rotation number. However an explicit description of these orbits is not given, while our stronger assumptions imply existence of simple periodic orbits and simple pairs of orbits. It may be true that the existence of any two periodic orbits with different rotation numbers implies the existence of a simple orbit with any given intermediate rotation number, but we feel that the evidence is not strong enough to make a conjecture either way, in particular in  view of Lemma 2.4. Even more difficult question is to determine whether there exists a simple pair of periodic orbits in the situation of Doeff Theorem. First one should try to find a pair of simple orbits which is not a simple pair while the rotation numbers are Farey neighbors, but it is very difficult to understand the geometry of the pseudo-Anosov homeomorphism which arises in this situation.

This research was originally motivated by a question we were asked by Professor Shmuel Fishman: Is there a topological explanation for the structure of appearance of accelerator modes in the kicked particle system.
In section \ref{global} we give a description of the kicked particle system. This system turns out to be described precisely by a family of shear homeomorphisms of the torus. The existence of accelerator modes is equivalent to existence of periodic orbits. The global structure of this system is given by the order relation in Theorem \ref{mainthm}, while it cannot be directly computed due to the complexity of the system.

The authors would like to thank Professor Shmuel Fishman for offering valuable insights, Professor Italo Guannieri for some critical advice, and Professor Philip Boyland for many  indispensable conversations.

\section{Simple orbits \label{simple section}}

Let $\{x_1,.....,x_N\}$ be a set of points belonging to one or more
periodic orbits for a homeomorphism $f$ of a surface $S$.
The dynamical properties of this set of orbits are captured by the induced action of $f$ on the complement
$S_0=S\setminus\{x_1,....,x_N\}$ in a sense that will shortly become clear. Choose any graph $G$ which is a deformation retract
of the punctured surface $S_0$. A homeomorphism of $S_0$ then induces
a map on $G$. The converse is also true: a given
map of $G$ determines a homeomorphism of $S_0$ up to isotopy.
Therefore we specify the
periodic orbits we analyze in terms of the action on a graph which
is a deformation retract of the surface after puncturing out the
orbit. 

Denote by $G_N$ the graph obtained by attaching $N$ small loops to the standard unit circle at the points $\exp(\frac{2\pi j}{N}i)$, $j=0,\hdots,N-1$.

\begin{Def} We call a periodic orbit $x=\{x_1,\hdots,x_N\}$ for a shear homeomorphism $f$ on the
two-torus a \emph{simple orbit} if the following hold. \ben

\item
There can be found a graph $G$ which is a deformation retract of $\T_0= \mathrm{T}^2\setminus x$ as on Figure \ref{simple} such
that $G$ is homeomorphic to $G_N$.
\begin{figure}[h]\centering
\includegraphics[width=5cm]{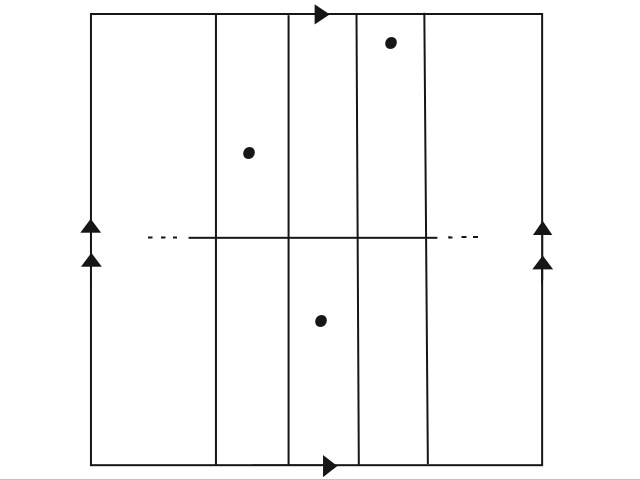}
\caption{A standard graph for a simple orbit}\label{simple}
\end{figure}\\
We call the loop in $G$ corresponding to the unit circle the \emph{horizontal loop} and the loops attached to it the \emph{vertical loops}.
\item There exists a homeomorphism $\tf$ of $\T_0$ isotopic to $f$ rel $x$ (i.e. the isotopy is fixed on $x$) so that a neighborhood of $G$ is invariant under $\tf$
and the induced action on $G$ satisfies: (a) There exists a fixed number $k\in\{0,..,N-1\}$ such that each vertical loop is mapped $k$ loops forwards (clockwise along the unit circle) to another vertical loop. (b)  The horizontal loop
is mapped to itself with one twist around one of the vertical loops.
\begin{figure}[h]\centering
\includegraphics[width=5cm]{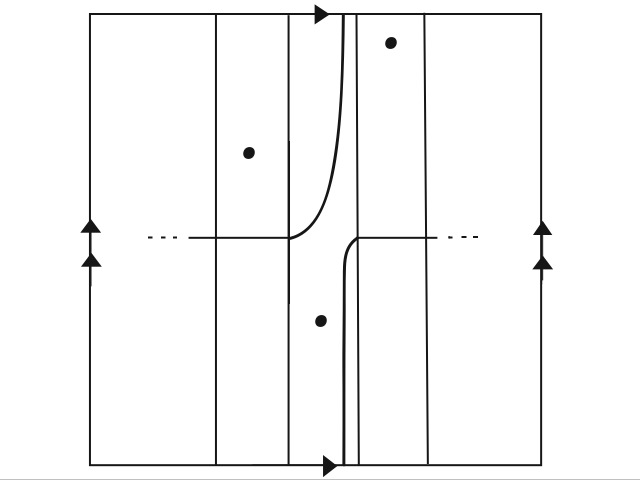}
\caption{The action on a standard graph for a simple orbit}\label{action}
\end{figure}\\
\een\end{Def}

\begin{rem} A homeomorphism $h$ for which we are given a simple
periodic orbit must be of shear type as we can deduce from the action
on the homology of the non-punctured torus. \end{rem}

For a shear homeomorphism there exists a basis for the first homology for which the induced map is represented by the matrix $\left(%
\begin{array}{cc}
  1 & 0 \\
  1 & 1 \\
\end{array}%
\right)$. 
From here on we refer to  any two axes given by an homology  basis that gives us the above representation as {\it standard axes}. The horizontal loop and one of the vertical loops in a graph for a simple orbits constitute a standard basis.

\begin{Def} 
 Let $h$ be a shear homeomorphism, and let $\hat{h}$ be a lift of $h$ to the universal cover (a plane). For any periodic point $x$ of $h$ of period $p$, $\hat{h}^p$ maps any lift $\hat{x}$ of $x$ the same integer number $q$ along the horizontal axis away from $\hat{x}$, in a standard choice of axis (and $\hat{x}$ is possibly mapped some integer number along the vertical axis as well). We can then define the \emph{rotation number} of $x$ to be $\rho(x)=\frac{q}{p} \mod1$.  The rotation number does not depend on the lift $\hat{h}$ of $h$. 

(The rotation number is often define relative to the given lifting of the homeomorphism $h$ and is not computed modulo 1, but we want it to depend only on the orbit and not on the lifting. In particular we want a simple orbit to have well defined rotation number independent of the lifting of $h$.)
\end{Def}

\begin{rem} In the case of a homeomorphism isotopic to a Dehn twist on a torus, which is our interest here, it can be easily shown that the abelian Nielsen type equals exactly the rotation number defined above. \end{rem}
 
There exists a simple orbit for any given rational rotation number $r\in [0,1)$, and it is unique up to similarity. Denote the similarity class by $\hat{r}$.
In the following we use the word vertical to describe the $y$ axis, in a standard choice of axis for $f$ (the direction along which the twist is made).

\begin{lem} \label{invariant loops} Let $x$ be a periodic orbit for a shear-type
homeomorphism $f$ of $\mathrm{T}^2$ for which there
exists a family of vertical loops such that they bound a set of
annuli each containing one point of the periodic orbit, and
this family is invariant under a homeomorphism $\tf$ isotopic to
$f$ rel $x$. Then $x$ is a simple orbit.
\end{lem}

\begin{proof} Choose a vertical loop $l$ of the invariant family. $f$ is orientation preserving, and so is $\tf$. The first loop to the right of $l$ is therefore mapped to the first loop to the right of $\tf(l)$. Hence, the vertical loops in the invariant family are all mapped the same number of loops to the right.

Now we have to find a horizontal line with the desired image. We
write the invariant family of loops as 
$\{l_i\}_{i=1}^{p}$, where $p=period(x)$, ordered along the horizontal axis. 
We choose another family of vertical loops $\{m_i\}_{i=1}^{p}$, such that $m_i$ is contained in the annulus between $l_i$ and $l_{i+1}$ ($l_{p+1}=l_1$), and passes through the periodic point $x_i$ also contained in this annulus. 
Choose a point $a_1\neq x_1$ on $m_1$. $f$ can be adjusted in such a way that the
new
homeomorphism $\tf$ leaves both families of vertical loops invariant, and in addition, so that $a_1$ be a periodic point of $\tf$ with period $p$. We denote the orbit of $a_1$ by $\{a_i\}_{i=1}^{p}$ where $a_i\in m_i$ for $1\leq i\leq p$. 

Choose a line segment $n_1$ connecting $a_1$ to $a_2$,
so it crosses the annulus between $m_1$ and $m_2$ from side to side. 
We choose $n_{j+1}$ to be the
line segment $\tf^j(n_1)$ for $1\leq j \leq p-2$. The boundary points of $n_j$ and $n_k$ coincide
whenever they lie on the same vertical loop.

Now, we look at the horizontal loop $n=\bigcup_{j=1}^{p}n_j$. Each segment of $n$ is mapped exactly to the next segment, except $n_{p-1}$ which is mapped into the annulus between
$m_1$ and $m_2$.
Since the mapping class group of an annulus is generated by a twist with respect to any  loop going once around the annulus we may assume, that $n_p$ is  mapped to $n_1$ plus a number of twists along such a loop. On the other hand we know that $f(n)$ is homotopic to itself plus one twist in the negative direction (on the closed torus), so $\tf$ maps n to itself plus one negative twist along this loop. By further adjustment of $\tf$ we may assume the twist is  made along $l_2$.
Thus the union of the vertical family $\{l_i\}$ with $n$ chosen as above constitute a graph showing $x$ to be a simple orbit.\end{proof}

The Thurston-Nielsen classification theorem, see \cite{thurston}, states that any homeomorphism $f$ on a closed connected oriented surface of negative Euler characteristic is isotopic to a homeomorphism $\tf$ which is\ben

\item pseudo Anosov, or

\item of finite order, or

\item reducible.

\een

where a homeomorphism $\phi$ is called \emph{of finite order} if
there exists a natural number $n$ such that $\phi^n=id$. A
homeomorphism $\phi$ is called \emph{pseudo-Anosov} if there
exists a real number $\lambda>0$ and a pair of transverse measured
foliations $(\mathcal{F}^u,\mathcal{\mu}^u)$ and
$(\mathcal{F}^s,\mathcal{\mu}^s)$ with
$\phi(\mathcal{F}^u,\mathcal{\mu}^u)=(\mathcal{F}^u,\lambda
\mathcal{\mu}^u)$ and
$\phi(\mathcal{F}^s,\mathcal{\mu}^s)=(\mathcal{F}^s,\frac{1}{\lambda}
\mathcal{\mu}^s)$. A homeomorphism $\phi$ on a surface $M$ is
called \emph{reducible} if there exists a collection of pairwise
disjoint simple closed curves $\Gamma=\{\Gamma_1,...,\Gamma_k\}$
in $int(M)$ such that $\phi(\Gamma)=\Gamma$ and each component of
$M\setminus \Gamma$ has a negative Euler characteristic.
The representative $\tf$ in the isotopy class of $f$ which is of one of
the three forms above is called the Thurston-Nielsen canonical
form of $f$. 

When the surface has a finite number of punctures and $\phi$ permutes the punctures then the same is true except that in the case of pseudo-Anosov map we treat the punctures as distinguished points (there is a unique way to extend a homeomorphism to the distinguished points) and we allow an additional type of singularities of the measured foliations, the 1-prong singularities at the distingushed points (See \cite{Hall}, and section 0.2 of \cite{BH}).

Of course homeomorphism  $f$ is reducible with respect to a simple
orbit since it contains an invariant family of loops $\{l_i\}$ and the
complement of the invariant family consists of punctured annuli
(which have negative Euler characteristic). 

\begin{rem} A homeomorphism $f$ with a simple orbit $x$ can be constructed in such a way that $x$ is the only periodic orbit of $f$. The invariant set of vertical loops is evenly spaced with the distance between the consecutive loops equal to 1/p. The loops are moved by q/p to the right and by fixed irrational number downward. Punctures (the points of the periodic orbit) are also evenly spaced and have the same height. The vertical lines containing punctures are moved by q/p to the right. The punctures  keep their height and all other points of the loop move a little downwards. Every other vertical line is moved to another vertical line by a little more than q/p to the right (not all lines by the same distance). 
\end{rem}

\begin{exm} \label{example}
Not every periodic orbit for a shear homeomorphism is reducible.
Consider the homeomorphism $h$ described on
Figure \ref{exm_fig}. It takes the graph on the left of Figure  \ref{exm_fig} to the graph on
the right and is a shear homeomorphism. It has a periodic orbit of
order 2, shown on the pictures, with rotation number $1/2$ and it is
pseudo-Anosov in the complement of the orbit.
\begin{figure}[htp]
\centering
\includegraphics[width=6cm]{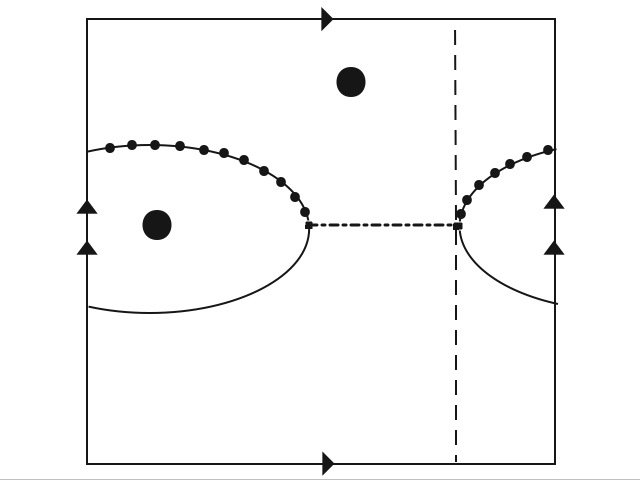}
\hfill
\includegraphics[width=6cm]{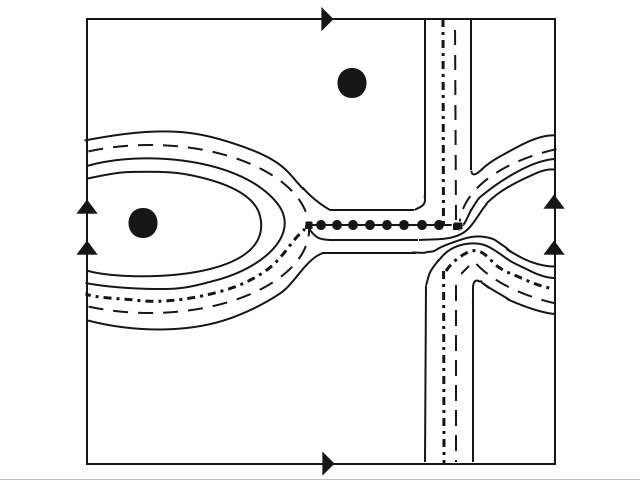}
\caption{}\label{exm_fig}
\end{figure}\\
\end{exm}

\begin{lem} \label{twist} There exists an orbit of twist type for a shear homeomorphism of the torus which is not simple. 
\end{lem}

\begin{proof}
We construct an example of an orbit of length 4 with the rotation number 1/2. It cannot be a simple orbit and yet we prove it does not force the existence of 
any periodic orbit not similar to itself, and is thus of twist type. Such examples may be known, possibly considered  for a different phenomena. We include it here in order to show the independence of our results.

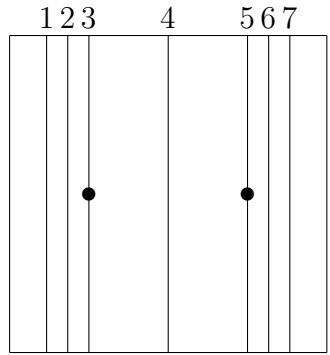
\begin{figure}
\centering

\begin{picture}(150,150)(20,-70)
\put(40,60){\line(1,0){120}}   \put(40,-60){\line(1,0){120}} 
\put(40,-60){\line(0,1){120}}   \put(160,-60){\line(0,1){120}} 
\put(54,-60){\line(0,1){120}}   \put(62,-60){\line(0,1){120}} 
\put(70,-60){\line(0,1){120}}   \put(100,-60){\line(0,1){120}} 
\put(130,-60){\line(0,1){120}}   \put(138,-60){\line(0,1){120}} 
\put(146,-60){\line(0,1){120}}  

 \put(51,63){1}     \put(59,63){2}    \put(67,63){3}   
 \put(97,63){4}    \put(127,63){5}    \put(135,63){6}   
 \put(143,63){7}   

 \put(70,0){\circle*{5}}
 \put(130,0){\circle*{5}}

\end{picture}

\caption{Some vertical lines in $U_2$}\label{step1}
\end{figure}

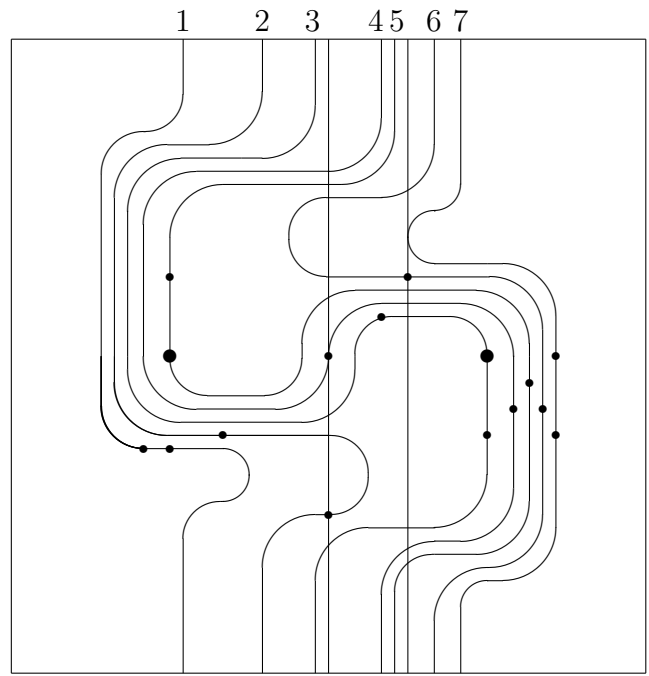
\begin{figure}
\centering

\begin{picture}(300,300)(20,-150)

\put(40,120){\line(1,0){240}}   \put(40,-120){\line(1,0){240}} 
\put(40,-120){\line(0,1){240}}   \put(280,-120){\line(0,1){240}} 
 \put(100,0){\circle*{5}}
 \put(220,0){\circle*{5}}

\put(125,0){\oval(70,40)[b]} \put(195,0){\oval(70,40)[t]}

 \put(125,0){\oval(50,30)[b]}   \put(120,-45){\oval(20,20)[r]}  
\put(195,0){\oval(50,30)[t]}   \put(193,0){\oval(86,50)[t]}  
\put(160,45){\oval(30,30)[l]}  \put(200,45){\oval(20,20)[l]}
\put(160,30){\line(1,0){60}}   \put(160,60){\line(1,0){20}} 
\put(90,-35){\line(1,0){30}}   \put(200,35){\line(1,0){20}} 

 \put(160,0){\oval(140,140)[tl]}   \put(160,0){\oval(120,130)[tl]}  
 \put(160,120){\oval(40,100)[br]}   \put(160,120){\oval(50,110)[br]}  
 \put(180,120){\oval(40,120)[br]}   \put(200,120){\oval(20,130)[br]}  
 \put(127,0){\oval(86,50)[b]}   \put(160,-45){\oval(30,30)[r]} 

   \put(127,0){\oval(86,150)[lt]}   \put(127,120){\oval(56,90)[rb]} 
 \put(100,0){\oval(42,160)[lt]}   \put(100,120){\oval(70,80)[rb]} 
 \put(90,0){\oval(32,170)[lt]}   \put(90,120){\oval(30,70)[rb]} 
 \put(160,0){\oval(162,60)[lb]}   \put(90,0){\oval(32,70)[lb]} 

\put(160,-120){\oval(50,120)[lt]}   \put(90,0){\oval(32,70)[lb]} 
\put(120,-120){\oval(30,130)[lt]}   \put(90,0){\oval(32,70)[lb]} 
\put(160,0){\oval(162,60)[lb]}   \put(90,0){\oval(32,70)[lb]} 

\put(200,0){\oval(40,130)[rb]}   \put(200,-120){\oval(90,110)[lt]} 
\put(200,0){\oval(60,140)[rb]}   \put(200,-120){\oval(40,100)[lt]} 
\put(200,0){\oval(72,150)[rb]}   \put(200,-120){\oval(30,90)[lt]} 
\put(220,-25){\oval(42,110)[r]}   \put(220,-120){\oval(40,80)[lt]} 
\put(220,-25){\oval(52,120)[r]}   \put(220,-120){\oval(20,70)[lt]} 

 \put(100,-35){\circle*{3}}    \put(90,-35){\circle*{3}}
 \put(120,-30){\circle*{3}}    \put(160,-60){\circle*{3}}
 \put(220,-30){\circle*{3}}    \put(180,15){\circle*{3}}
 \put(230,-20){\circle*{3}}    \put(160,0){\circle*{3}}
 \put(100,30){\circle*{3}}    \put(236,-10){\circle*{3}}
 \put(241,-20){\circle*{3}}    \put(190,30){\circle*{3}}
 \put(246,0){\circle*{3}}    \put(246,-30){\circle*{3}}
 
\put(160,-120){\line(0,1){240}}   \put(190,-120){\line(0,1){240}} 

 \put(102,123){1}    \put(132,123){2}
 \put(151,123){3}    \put(175,123){4}
 \put(183,123){5}    \put(197,123){6}
 \put(207,123){7}

\end{picture}

\caption{The foliation in $U_2$ realizing a non-simple orbit of twist type.}\label{twist}
\end{figure}

 We represent the torus as the unit square with the opposite sides identified. The points $A_1,A_2,A_3,A_4$ of the orbit are spaced evenly on the horizontal middle line with the x-coordinate 1/8, 3/8, 5/8, 7/8.
We split the square into 2 equal parts $U_1$ and $U_2$ by the vertical line $x=1/2$. Homeomorphism $h$ translates $U_1$ to the right to $U_2$. Vertical lines go to vertical lines, lines $x=0$ and $x=1/2$ move downward by an irrational number $\alpha<1/40$ and the movement is damped out to the horizontal translation for $t<\alpha$ and $t>1/2-\alpha$, so the other vertical lines are translated horizontally by 1/2. In particular $A_1$ moves to $A_3$ and $A_2$ moves to $A_4$. 

The restriction of $h$ to $U_2$ is defined in two steps. The second step simply translates $U_2$ horizontally by 1/2 to the right (which is the same as the translation by 1/2 to the left). The first step is isotopic to the half-twist along the segment connecting $A_3$ and $A_4$, followed by the Dehn twist with respect to the right side (right boundary of the cylinder). In particular it switches $A_3$ and $A_4$. We shall prove that we can construct such $h$ for which $h^2$ has no fixed points and therefore $h$ has no periodic orbit of length 2 and in particular no simple orbit with the rotation number 1/2.

We describe the first step of $h$ restricted to $U_2$. Figure \ref{step1} shows some vertical lines in $U_2$, the big dots show the points $A_3$ and $A_4$ of the periodic orbit. Figure \ref{twist} shows their images under the first step. In these pictures $U_2$ is represented as a square, to make more space, but in the reality the base has length 1/2 and the height is
equal to 1. Line $x=1/2$ is mapped to itself and moves downward by $\alpha$. The near by vertical lines (for $t<1/2+1/30$) are moved to the vertical lines and to the right where line $x=1/2+1/30$ is moved to the line $x=1/2+1/20$. For 
$t\in (1/2+1/30,1-1/20)$ the line $x=t$ is moved to a curve $L_t$ and for 
$t\in (1-1/20,1)$ the line $x=t$ moves to a vertical line to the right of it and downward, to get the full Dehn twist plus a movement downward by $\alpha$ when we get to the line $x=1$. For $t\in (1/2+1/30,1-1/20)$ curve $L_t$ starts at a point on the top side to the right of $x=t$, it moves to the left, then to the right, then to the left again and ends at the bottom side (exactly below its starting point). In particular each vertical line meets $L_t$ in at most two points. Some lines $L_t$ are shown on Figure \ref{twist}.

 We may arrange it in such a way that there exist $t_0,t_1$ such that $1/2<t_0<t_1<1$ and the line $x=t$:

 is disjoint from $L_t$, and lies on the left side of $L_t$ when $t<t_0$;

 meets $L_t$ at one point for $t=t_0$;

meets $L_t$ at two points when $t_0<t<t_1$;

meets $L_t$ at one point for $t=t_1$;

is disjoint from $L_t$ and lies on the left side of $L_t$ when $t_1<t<1$.

We get a new trivial foliation of the annulus $U_2$. In step 1 we map the vertical foliation onto the new foliation $L_t$. We can further change the first step moving each leave $L_t$ along itself to reach the following goal. 
Let $P_t,Q_t$ denote the intersection points of $x=t$ with $L_t$, $P_t$ lies below $Q_t$ (the points coincide for $t_0$ and $t_1$). For $t=t_0$ the line $x=t$ meets $L_t$ in one point $P_t$. We may assume that the image of $P_t$ in $L_t$ lies in the part below $P_t$. Then for the nearby leave the images of both points $P_t$ and $Q_t$ lie in the lower part of $L_t$ below the point $P_t$ (see the small dots on the first curve in Figure \ref{twist}). The images of $P_t$ and $Q_t$ lie further away from each other when we move to the right (see the small dots on the second curve). The third line passes through $A_3$, its image $L_t$ passes through $A_4$  and the images of $P_t$ and $Q_t$ lie on different sides of $A_4$ along the third curve. Next the upper point $Q_t$ moves backwards along $L_t$ and when we reach the fourth line of Figure \ref{step1} (also shown on Figure \ref{twist}) it coincides with the point $P_t$ on $L_t$. Next the image of $Q_t$ lies inside the arc of $L_t$ between the points $P_t$ and $Q_t$ and when we reach the fifth line on Figure \ref{step1}, which passes through the point $A_4$, then the curve $L_t$ passes through $A_3$ and the image of $Q_t$ on $L_t$ lies above $A_3$ (see Figure \ref{twist}). When we move further to the right the image of the point $Q_t$ moves again forward towards the image of $P_t$ and at the line number 6 on Figure \ref{step1} the image of $Q_t$ again coincides with $P_t$ at the intersection of $x=t$ with $L_t$. Next the images of $P_t$ and $Q_t$ move further down and gets close together and when $t=t_1$ we have one intersection point $P_t$ and its image lie below $P_t$ along $L_t$.
Step 1 has no fixed points. Step 2 translates $U_2$ to $U_1$.

We now consider the homeomorphism $h^2$. We start with $U_1$. Any point on $x=0$ and $x=1/2$ moves down by $2\alpha$. Any point with $x\in [0,1/30]$ moves to a point with a bigger $x$-coordinate. Any point with $x\in [1/30,9/20]$ moves horizontaly by 1/2 then we apply step 1, which has no fixed points, and then the point moves again horizontaly by 1/2 so it comes to a new point. Any point with $x\in [9/20,1/2]$ moves to a point with a bigger $x$-coordinate.

For points in $U_2$ the situation is similar. Any point with $x\in (1/2,16/30]$ moves to a point with a bigger $x$-coordinate. Any point with $x\in [16/30,19/20]$ moves under the first step to a new point with the $x$-coordinate in $[11/20,29/30]$ and then moves horizontaly twice by 1/2. Finally any point with $x\in [19/20,1]$ moves to a point with a bigger $x$-coordinate. Homeomorphism $h^2$ has no fixed points and $h$ has no periodic points of order 2.

We now show that there exists a homeomorphism $f$ isotopic to $h$ in the complement of the orbit $A_1,A_2,A_3,A_4$, which has only periodic orbits similar to this orbit and periodic orbits of order 2. We consider parts $U_1$ and $U_2$ as before. The restriction of $f$ to $U_1$ translates it horizontaly by 1/2. In $U_2$ we choose two circles with center $(3/4,1/2)$ and radius 1/7 and 1/6 respectively.  We rotate the interior of the smaller circle by 180 degrees. The rotation is damped out to the identity at the outer circle and the intermediate circles are moving out towards the outer circle. The exterior of the outer circle with $x<19/20$ is pointwise fixed. The lines with $x>19/20$ move to the right and down  to get the full Dehn twist when we get to the line $x=1$. The second step of $f$ restricted to $U_2$ translates it horizontally by 1/2. Now each point inside the smaller circle, different from its center (which has period 2), belongs to an orbit similar to  $A_1,A_2,A_3,A_4$. Points between the circles  and points with $x\in (19/20,1)$ are not periodic and other points in $U_2$ have period 2 and the same is true for the corresponding points in $U_1$. 

Therefore the orbit $A_1,A_2,A_3,A_4$ does not force any periodic orbit not similar to itself. 
\end{proof}

\section{Simple orbit pairs}\
\label{pairs}
Let $x$ and $y$ be two coexisting simple periodic orbits, for a
homeomorphism $f$ of $\T^2$ 
($f$ must be of shear type),
with rotation numbers $\frac{q_1}{p_1}$ and $\frac{q_2}{p_2}$
respectively. Assume $p_1>p_2$, i.e., $y$ has lesser period than $x$.
\begin{Def} \label{simple pair}  We call the pair of orbits a {\emph simple pair} if 
\begin{itemize}

\item
We can find an embedded graph $G$ in their complement homeomorphic to $G_{p_1}$ as on Figure \ref{A}. 
\begin{figure}[h]
\includegraphics[width=5cm]{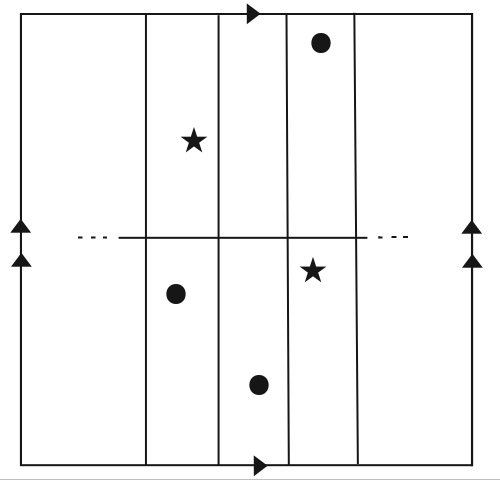}
\caption{}\label{A}
\end{figure}\\
Each component in the complement of the graph is a topological
rectangle which contains exactly one point of orbit $x$ and at
most one point of orbit $y$.

\item The homeomorphism $f$ acts on this graph
in the following way: each vertical loop except one moves to
another vertical loop, there is one vertical loop denoted $l$ such
that $f(l)$ is a vertical loop $m$ plus a small loop around one
point of the (shorter) $y$ orbit, in a rectangle adjacent to line
$m$ on the right (as on Figure \ref{B}) or on the left, and the
horizontal line is mapped to itself plus a twist in the negative
direction around $f(l)$, as on Figure \ref{B}.
\begin{figure}[h]
\includegraphics[width=5cm]{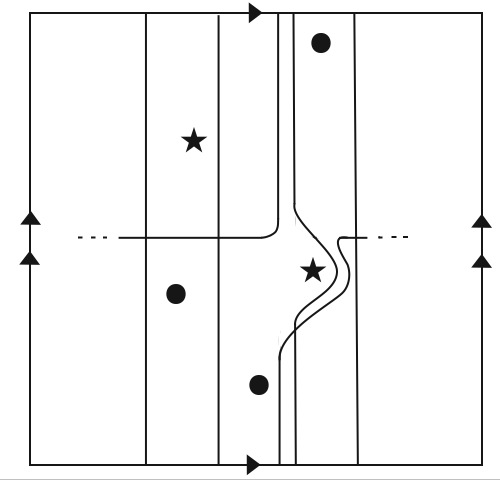}
\caption{}\label{B}
\end{figure}
\end{itemize}\end{Def}
The graph which appears in Definition \ref{simple pair} divides
the torus $\T^2$ into $p_1$ rectangles. The homeomorphism $f$
moves each vertical loop the same distance, say $k$ rectangles, to the
right except for the small additional loop for line $l$. Let $R_0$
be the rectangle adjacent to $m$ in which the small loop in the
image of the graph occurs.
$R_0$ must contain exactly one point of each orbit. We denote these points
$x_0$ and $y_0$ respectively. Under $p_1$ iterations of $f$
the point $x_0$ runs $q_1$ times around the whole torus, that is $q_1p_1$
rectangles to the right. So,
$p_1k=q_1p_1$ and $k=q_1$. 

The point $y_0$  is mapped to
itself after $p_2$ iterations.Under each iteration the image of
$y_0$ is mapped $k\ (=q_1)$ rectangles to the right, except the last iteration under which it is moved an additional rectangle to the right or left. Altogether it has moved  $p_2q_1\pm 1$ rectangles. 
At the same time, it is mapped around the torus $q_2$
times, hence $q_2p_1$ rectangles. This means $p_2q_1\pm 1=q_2p_1$. Thus for a
simple pair of periodic orbits the rotation numbers $\frac{q_1}{p_1}$ and
$\frac{q_2}{p_2}$ are Farey neighbors and the additional loop is
on the right (as on Figure \ref{B}) if and only if
$\frac{q_2}{p_2}>\frac{q_1}{p_1}$. 
Denote by $\wh{r}\vee\wh{s}$ the similarity class of a simple pair corresponding to a pair of Farey neighbors $r\vee s$.

Consider again the points $x_0$ and $y_0$ in the rectangle $R_0$.
Continue the notation to all the points of $x$ and $y$ by $x_i=f^i(x_0)$ and
$y_i=f^i(y_0)$. We draw a small loop around each of the points
of $y$. The union of these loops will be
the peripheral subgraph $P$ for the Bestvina-Handel algorithm, since we may assume the union of these loops to be $f$-invariant. Now we
consider separately two cases. Case 1 will be the case in which $m$
is the left boundary curve of $R_0$, while in case 2 it is the
right boundary (in other words in the first case
$\frac{q_1}{p_1}<\frac{q_2}{p_2}$ and in the second case
$\frac{q_1}{p_1}>\frac{q_2}{p_2}$). Choose some point on the loop
around $y_0$ and connect it, by a curve $l_0$, to a
point on the section of the horizontal line in $R_0$, in case 1
from below the segment and in case 2 from above.
\begin{figure}[h]
\includegraphics[width=5.5cm]{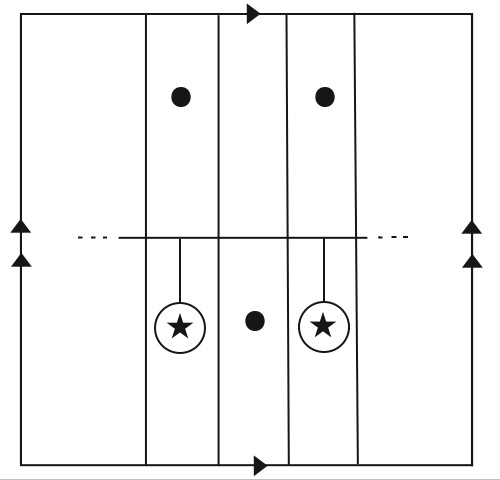}
\caption{}\label{C}
\end{figure}

Then 
$f(l_0)$ is a curve connecting the loop around $y_1$ and the
corresponding horizontal segment. We denote it by $l_1$, and do the same
for each $y_i$. 
After adding the above edges to the graph case 1 is topologically as in Figure \ref{C}. 

The inclusion
$G\hookrightarrow S_0$ is a homotopy
equivalence (where $S_0$ is the punctured torus). We know
the action of $f$ on all edges of $G$ except for the curve $l_{q_{1}-1}$ connecting $y_{q_1-1}$ and the horizontal
segment in the corresponding rectangle. It's image is a curve
connecting the horizontal segment in the rectangle
adjacent to $m$ which is not $R_0$ to the loop around $y_0$. This image might wind around a disk containing
$y_0$ and $x_{q_1-1}$ as in Figure \ref{winding}
\begin{figure}[h]
\includegraphics[width=3cm]{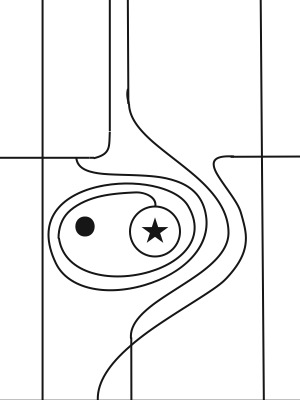}
\caption{}\label{winding}
\end{figure}\\
The graph and its image for case 2 are exactly the same except
that the loops are connected to the horizontal segments from
above. We shall prove in Proposition \ref{no
twists} that we may assume that the image of the segment
$l_{q_1-1}$ has no winding. Hence we draw from now on the graph images
without winding, and we may assume the graphs given in Figure \ref{pair} also
have an invariant neighborhood by a further isotopy of $f$.
\begin{figure}[h]
\includegraphics[width=6cm]{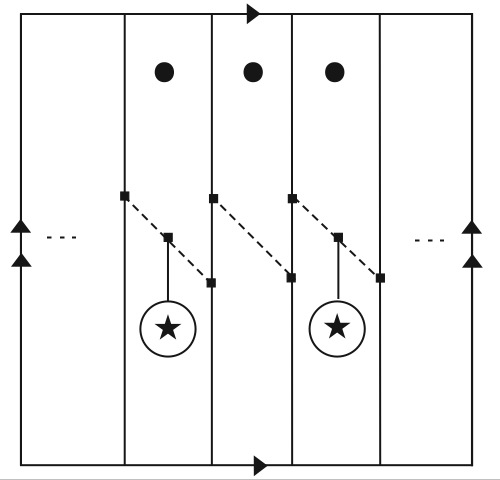}
\includegraphics[width=6cm]{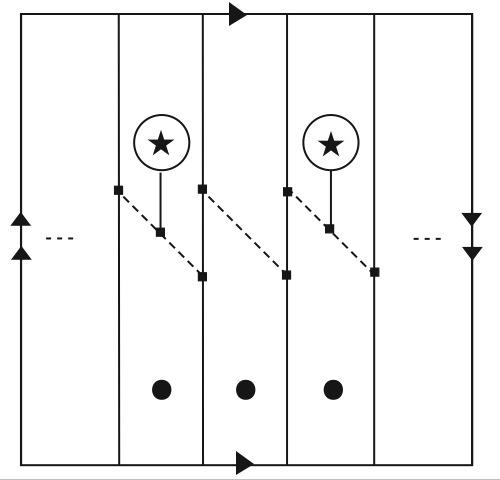}
\caption{ The standard graph for a simple pair, case 1 on the left and case 2 on the right.}\label{pair}
\end{figure}\\
The action of $f$ (up to isotopy) on this graph is given by one of
the actions on Figure \ref{pair_im}, drawn in some regular neighborhood of
the graph, where each vertical loop moves $q_1$ loops
to the right. 
\begin{figure}[h]
\includegraphics[width=6cm]{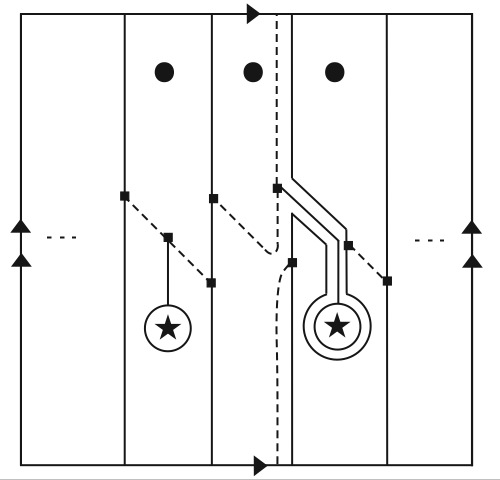}
\includegraphics[width=6cm]{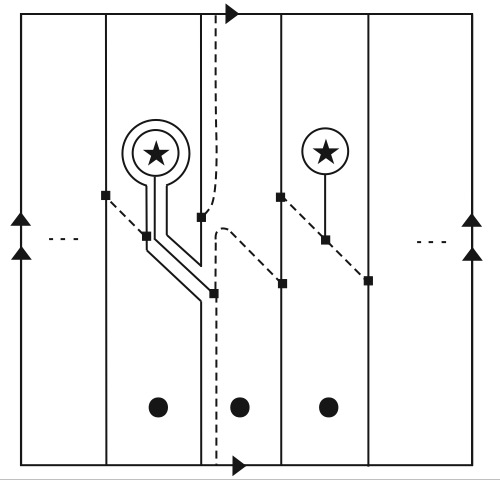}
\caption{The action on a standard graph for a simple pair, for both cases respectively.}
\label{pair_im}
\end{figure}
\begin{prop} Let $\{x,y\}$ be a simple pair with the graph as on Figure \ref{C} and with windings as on
Figure \ref{winding}. Then there may be chosen a
different invariant graph, which also makes $\{x,y\}$ a simple
pair, whose image is without winding. \label{no twists} \end{prop}

\begin{proof}
To simplify the picture we prove the proposition for
rotation numbers $\frac{1}{3}$ and $\frac{1}{2}$ . The
general proof proceeds in the same way. We start by looking at a
simple pair $\{x,y\}$ of rotation numbers $\frac{1}{3}$ and
$\frac{1}{2}$ for a homeomorphism $f$ with the corresponding
invariant graph $G$ given so that the action on it is without any
twists, as on the left side of Figures \ref{pair} and \ref{pair_im}.  We now
choose a different system of curves (a different graph), in a
small neighborhood of $G$ as in Figure \ref{E}, which will serve as a new graph for the pair.
\begin{figure}[h]
\includegraphics[width=6cm]{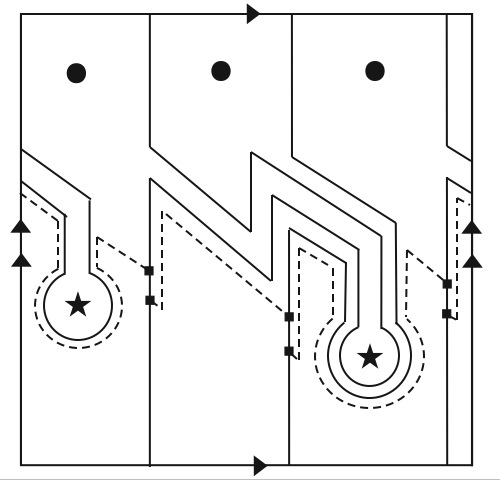}
\caption{}\label{E}
\end{figure}\\
Solid lines are the new vertical loops and dashed lines are the
    new diagonal segments like in figure \ref{pair}. The horizontal loop
    consists of the dashed lines and the long pieces of the solid
    lines. To move from left to right along the horizontal line, move along a dashed line and turn
    to the left when meeting a solid line. Continue up
    along a vertical loop and then along the next dashed line.
We add to this graph the peripheral subgraph and the connecting
   segments and get the graph $H$ as on Figure \ref{new1}. It is clear that
   topologically the graph $H$ has the same form as the graph on
Figure \ref{pair} and that it has an invariant neighborhood.
\begin{figure}[h]
\includegraphics[width=6cm]{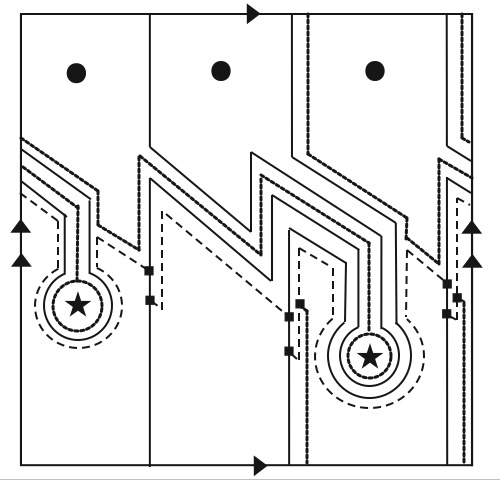}
\caption{}\label{new1}
\end{figure}\\

The reader can check (using the precise knowledge
of the image of each edge of the original graph) that the action
of $f$ on the graph $H$ has the properties required from a simple
pair. Each vertical loop is mapped onto another vertical loop
except for one loop $l$ for which $f(l)$ is equal to a loop $m$
plus a loop around the next periodic point $y_p$ of the shorter
orbit. The horizontal loop is mapped onto itself plus a negative
Dehn twist along $f(l)$. Consider the image $s$ of the segment
which connects the horizontal loop to the periodic point
$y_{p-1}$. When the action has no twist then $s$ moves along the
horizontal loop in its positive direction until it meets the
original segment connecting to $y_p$ and then it follows along the
segment. However in our case $s$ goes first backwards along the horizontal
loop than moves in the counterclockwise direction along the
boundary of the "rectangle" adjacent to the vertical loop $m$ and
finally follows the horizontal loop and the segment to $y_p$. This 
means that the action $f$ on the graph $H$ has one positive twist.

  We proved that a simple pair for a shear homeomorphism with a
  given graph and a given action without twists can be given
  another graph which also describes it as a simple pair and the
  action on the new graph has one positive twist. This process is
  reversible. Therefore, by induction, we can add or remove any
  number of twists using a suitable graph. This implies Proposition
  \ref{no twists}.
\end{proof}

Hence for a simple pair the action on a spine is given by Figures \ref{pair} and \ref{pair_im}. We can now apply the Bestvina-Handel
algorithm (see \cite{BH}), endowing a neighborhood of $G$ with a fibered structure in the natural way. The algorithm specifies a finite number of steps which we apply to the graph $G$, altering $G$ together with the induced action on it, but without changing the isotopy class of $f$ on $\T^2\setminus({x}\cup{y})$. When the algorithm terminates, it gives a
new homeomorphism $\tf$ 
which is the Thurston Nielsen canonical form of $f$. 

For simple pairs, the action in each of the two cases above is
easily seen to be tight, as no edge backtracks and for every vertex there are two edges whose images emanate in different directions. The action has no invariant non-trivial forest
or nontrivial invariant subgraph and the graphs have no valence 1
or 2 vertices. 
This is the definition in \cite{BH} for an
irreducible map on a graph.

\begin{Def}
Assuming $g$, the induced map on the graph
itself, does not collapse any edges, there is an induced map $Dg$,
the \emph{derivative} of $g$, defined on 
\begin{eqnarray*}
\mathcal{L}:=\coprod\{(v,e) |v \text{ is a vertex of } G, e \text{ is an oriented edge emanating from } v\}
\end{eqnarray*}
by $Dg(v,a)=(g(v),b)$
where $b$ is the first edge in the edge path $g(a)$ which emanates
from $g(v)$.
\end{Def}

\begin{Def} We say two elements $(v,a)$ and $(v,b)$ in $\mathcal{L}$
corresponding to the same vertex $v$ are \emph{equivalent} if they
are mapped to the same element under $D(g^n)$ for some natural n.
The equivalence classes are called \emph{gates}\end{Def}

The gates in each of the cases above are given by Figure \ref{gates}, indicated there by small arcs.
\begin{figure}[h]
\includegraphics[width=6cm]{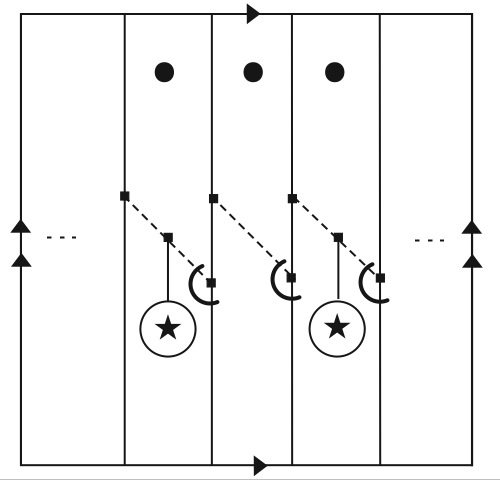}
\includegraphics[width=6cm]{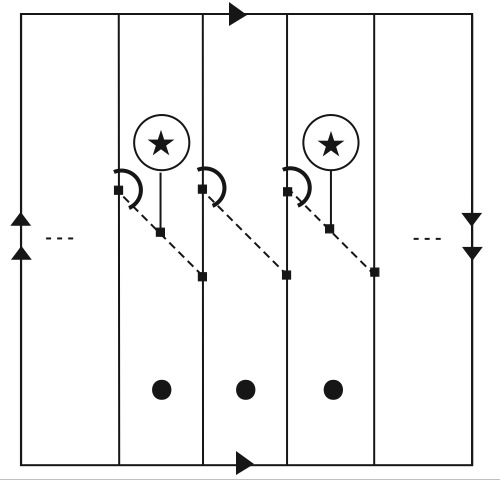}
\caption{}\label{gates}
\end{figure}\
There is no edge which $g$ sends to an edge
path which passes through one of the gates - enters the junction
through one arm of the gate and exits through the other. Such an irreducible map is efficient. i.e., this is an end point of the algorithm. Now, since
there are edges mapped to an edge path longer than one edge, we arrive at our next theorem.

\begin{thm} 
A homeomorphism $f$ of the two torus for which a simple pair of periodic orbits exists is isotopic to a pseudo-Anosov homeomorphism relative to this pair of orbits.
\end{thm}

Let $f$ be a shear type homeomorphism of the torus, and fix a lift $\tf$ of $f$.
Define the \emph{lift rotation number} of a point $x\in\mathrm{T}^2$ to be
\begin{eqnarray*}
\rho(x,\tf)=lim_{n\rightarrow\infty}\frac{(\tf^n(\hat{x})-\hat{x})_1}{n},
\end{eqnarray*}
for any lift $\hat{x}$ of $x$, when the limit exists, where the subscript 1 denotes the projection to the horizontal axis. Define the \emph{rotation set} $\rho(\tf)$ of $\tf$ to be the set of accumulation points of 
\begin{eqnarray*}
\bigg\{\frac{(\tf^n(\hat{x})-\hat{x})_1}{n}|\hat{x}\in\mathbb{R}^2 and\ n\in\mathbb{N}\bigg\}
\end{eqnarray*}

Then, the above theorem follows  from the following much more general theorem by 
Doeff, see \cite{Doeff} and \cite{DM}. 

\begin{thm}\begin{bf}{(Doeff)}\end{bf} Let $h$ be a shear type homeomorphism of $\mathrm{T}^2$, and fix a lift $\tilde{h}$ of $h$. If $h$ has two periodic points $x$ and $y$ with $\rho(x,\tilde{h})\neq\rho(y,\tilde{h})$ then $h$ is pseudo-Anosov relative to $x$ and $y$.
Furthermore, the closure of the rotation set is a
compact interval, and any rational point $r$ in the interior of
this interval corresponds to a periodic point $x\in\mathrm{T}^2$ with 
$\rho(x,\tilde{h})=r$.
\label{Doeff's theorem} \end{thm}

In particular, Doeff proves existence of two periodic orbits of different rotation numbers
implies existence of an orbit for any rational rotation number between
these two. But he does not give any characterization of these
orbits. Example \ref{example} shows two different orbits, both
with rotation number equal to 1/2, one of which is
pseudo-Anosov, and the other reducible. Thus the rotation number does not give much information about the orbit and in this
sense this theorem does not give a satisfactory dynamical
understanding of what is happening in regions of coexistence of
orbits. In contrast with Doeff's general theorem, we get results for a very specific family of periodic orbits,
but for this family we are able to give exactly the orbits forced
by others, as we show in section \ref{order}. 

In our case
the canonical form $\tf$ of $f$ we get by applying the Bestvina-Handel algorithm is a pseudo Anosov homeomorphism.
When this is the case, the algorithm
gives a canonical way of endowing a regular neighborhood $S_0$ of $G$ with a
rectangle decomposition $\{R_1,...,R_N\}$. The decomposition is a Markov
partition for the homeomorphism $PSL$.

A Markov partition for a dynamical system offers a symbolic representation for
the system in the following way. Let
$\Sigma_N$ be the subset of the full $N$-shift (the set of
bi-infinite series on $N$ symbols), where $N$ is the number of rectangles in the decomposition. Let $\Sigma$ be a subset of $\Sigma_N$ defined by \bc
$\Sigma=\{s=(...,s_n,s_{n+1},....): R_{s_n}\cap \tf^{-1}
R_{s_{n+1}}\neq \emptyset\}$ \ec

On $\Sigma$ we naturally define a dynamical system with the
operator of the right shift denoted by $\sigma$, and $(\Sigma,
\sigma)$ is called the subshift corresponding to the dynamical
system. $\Sigma$ can be completely described by stating which
transitions $k\rightarrow m$ for $k,m \in \{1,....,N\}$ are
\emph{allowed} (i.e., for which $k,m$, $\tf^{-1}R_m \cap R_k \neq
\emptyset$). See \cite{Adler} for the definitions and
for a proof that in this case we can define a map
$\pi:\Sigma\rightarrow S_0$ by \bc
$\pi:s\mapsto\bigcap_{n=0}^{\infty}\overline{\tf^nR_{s_{-n}}\cap....\cap\tf^{-n}R_{s_n}
}$ \ec
which satisfies the following properties:
\begin{itemize}

\item $\pi\sigma=\tf\pi$,

\item $\pi$ is continuous,

\item $\pi$ is onto.

\end{itemize}

We take here the set of
sequences with the Tichonoff topology. Thus a periodic point in the symbolic dynamical system which is just a periodic sequence corresponds
to a periodic point in the original dynamical system.

To obtain the Markov partition in our case as in \cite{BH}, we thicken the edges of the graph to rectangles. In particular, the rectangles can be glued directly to each other without any junctions. This can be done in a smooth way, endowing $S_0$ with a compact metric space structure by giving a length and width to each rectangle, consistently. Each edge of the standard graph for the pair (figure \ref{pair}) corresponds to
one rectangle, except the edge which is mapped to the
loop around $y_0$. This edge we divide in two (this is necessary
to avoid having a rectangle intersecting twice an
inverse image of another rectangle). Now we have edges of 7
different types on the graph. The vertical loops of the graph
consist of long edges we denote as $A$ edges, and short edges we
call $B$ edges. The loops around the points of the $y$ orbit and vertical segments connecting the loops
to the diagonal edges we call $C$'s and
$D$'s respectively. In rectangles which contain two punctures and
therefore two diagonal edges we call the upper ones $L$ edges and
the lower ones $K$ edges in the first case, and the lower ones $L$ edges, upper ones $K$ edges, in the second case. The last type of edges are diagonals of
once punctured
rectangles, these we call $M$ edges.
\begin{figure}[h]
\includegraphics[width=6cm]{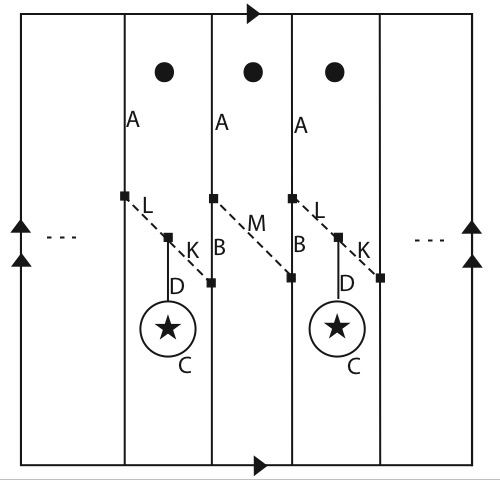}
\caption{}
\end{figure}\\

Next, we label the rectangles in order to have explicitly the
transition rules:
\begin{itemize}

\item For $0\leq i \leq p_2-1$ denote the rectangle corresponding
to the $D$ edge connecting the loop around $y_i$ to the diagonal
by $r_{i+1}$. Denote the rectangle corresponding to the $C$
edge which is the loop around $y_i$ by $r_{p_2+i+1}$.

\item For $1\leq i \leq p_2$, denote the rectangles corresponding
to the $L$ and $K$ edges connected to $r_{i}$ by $r_{2p_2+i}$
and $r_{3p_2+i}$ respectively.

 \item Denote the rectangle corresponding to the $A$ edge belonging
to the vertical line we referred to as m by $r_{4p_2+1}$ and the
$B$ edge which is part of the same line m as $r_{4p_2+p_1+1}$.

\item For the vertical line $f^i(m)$ denote it's $A$ and $B$
rectangles by $r_{4p_2+1+i}$ and $r_{4p_2+p_1+1+i}$ respectively
for all $1\leq i \leq p_1-2$

\item For the vertical line $f^{p_1-1}(m)$, denote it's $A$ edge
as by $r_{4p_2+p_1}$. There are two rectangles corresponding to
the $B$ edge as explained above, denote the lower one by
$r_{4p_2+2p_1}$ and the upper one by $r_{4p_2+2p_1+1}$.

\item Label the $p_1-p_2$ remaining rectangles corresponding to
the $M$ edges by starting with the first of these to the right of
m, and then continuing by the order along the horizontal axis,
denoting them by $r_{4p_2+2p_1+2}$, ... ,$r_{3p_1+3p_2+1}$
\end{itemize}

Finally, we can look at the diagram in figure \ref{diagram}, showing the set of rectangles and transitions in this Markov partition which we now use.
\begin{figure}[h]
\includegraphics[width=7cm]{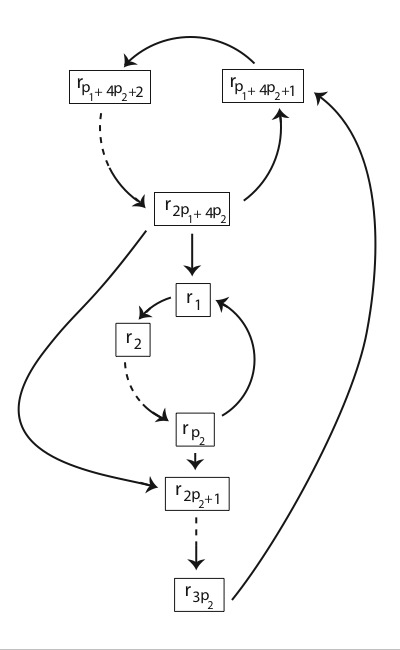}
\caption{Some of the rectangles in the Markov partition, where the arrows
denote allowed transitions between them}\label{diagram}
\end{figure}

A periodic symbolic sequence of
allowed transitions gives as explained a
periodic point in the original dynamical system. Therefore by this
diagram we can easily find other periodic orbits on the torus that
must exist for $f$. We will later prove that these orbits are in fact simple, but this 
will require some more work.
Hence, by this diagram we prove only existence of orbits with specified rotation numbers.
For every pair $(n,m)$ of natural numbers,
$n,m\neq 0$, by starting from $r_{p_1+4p_2+1}$, going n times
around the first loop in the diagram $\{r_{p_1+4p_2+1},\dots, r_{2p_1+4p_2}\}$,
then going m-1 times around the second loop $\{r_1,\dots, r_{p_2}\}$
(and skipping it if m=1) and then returning through the final
sequence $\{r_{2p_2+1},\dots,r_{3p_2}\}$ to $r_{p_1+4p_2+1}$, we get
a periodic symbolic allowed sequence, and so a new periodic orbit we denote $O_{n,m}$. These symbolic sequences are all different and hence so are the periodic orbits. We look at a point $p\in O_{n,m}$ such that $p$ is in the rectangle $r_{p_1+4p_2+1}$.
For the first $n\cdot{p_1}$ iterations of $p$, corresponding to each time the upper loop in the diagram appears in the symbolic sequence of $p$, the images are contained in the $B$ edges. 
The vertical loops are mapped under $f$ retaining the same
"horizontal distance" from the periodic points from the $x$ orbit to
their left. So, $p$ is mapped a total distance of $n\cdot q_1$ along the horizontal axis under $f^{n\cdot p_1}$.

Similarly, point $f^{np_1}(p)$, which lies in the rectangle $r_1$
corresponding to $D$ edge, is mapped a distance $q_2$ along the
horizontal axis under each iteration of $f^{p_2}$, for every
occurrence of the second loop in the symbolic sequence of $p$. This
is because the $D$ edges retain their distance from the $y$ orbit
points below them. The final sequence in the symbolic representation
of $p$ until the return to the first loop also corresponds to the
horizontal distance $q_1$. These last points of the periodic orbit lie
in rectangles corresponding to $L$ edges.
So $p$ is mapped a horizontal distance of $nq_1+mq_2$ under $f^{np_1+mp_2}$. Hence, the new orbit $O_{n,m}$
has rotation number $\frac{nq_1+mq_2}{np_1+mp_2}$. 

See \cite{Farey} for a proof that any two Farey neighbors span this way all rational numbers between them, that is all rationals between $\frac{q_1}{p_1}$ and 
$\frac{q_2}{p_2}$ are of the form $\frac{nq_1+mq_2}{np_1+mp_2}$. So we found a periodic orbit of any rational number between the two original rotation numbers $\frac{q_1}{p_1}$ and $\frac{q_2}{p_2}$.

Note we have found these simple periodic orbits for the Thurston-Nielsen canonical
form of the homeomorphism $f$ we started with. It remains to relate these periodic orbits to the periodic orbits of $f$ itself. Recall the following definition from \cite{AF}.

A periodic point $x_0\in S$ of period $p$ for homeomorphism $f_0$ is called
\emph{unremovable} if for each given homomorphism $f_1$ with $f_t:f_0\simeq f_1$
there is a periodic point $x_1$ of period $p$ for $f_1$and an arc
$\gamma:[0,1]\to S$ with $\gamma(0)=x_0$, $\gamma(1)=x_1$ and $\gamma(t)$ is a periodic point of period $p$ for $f_t$.

It was proven by Asimov and Franks in \cite{AF} that every periodic orbit of a pseudo-Anosov diffeomorphism is unremovable. Thus orbits found for the pseudo-Anosov representative exist for any other homeomorphism in its isotopy class.
This yields all these periodic
orbits exist for the original homeomorphism $f$ as well. Thus we get theorem \ref{Doeff's theorem} for our
specific case:

\begin{thm}\label{con:infty} If there exists a simple pair of orbits for a
homeomorphism $f$ of the torus of abelian Nielsen types $s$
and $t$ which are Farey neighbors, there exists a periodic
orbit for $f$ with abelian Nielsen type equal to $r$ for every
rational number $r$ between $s$ and $t$.
\end{thm}

\section{The order relation}\
\label{order}
For any simple pair $\wh{\frac{q_1}{p_1}}\vee \wh{\frac{q_2}{p_2}}$, The orbit $O_{1,1}$ out of the family of new orbits we constructed above has rotation number equal exactly to $\frac{q_1+q_2}{p_1+p_2}$. This orbit corresponds to the
symbolic sequence $r_{p_1+4p_2+1}\rightarrow
r_{p_1+4p_2+2}\rightarrow....\rightarrow r_{2p_1+4p_2}\rightarrow
r_{2p_2+1}\rightarrow....\rightarrow r_{3p_2}\rightarrow
r_{p_1+4p_2+1}$ as in the diagram in figure \ref{diagram}. So we have a list of rectangles,
each containing exactly one periodic point from the new orbit $O_{1,1}$. 
We denote the point of $O_{1,1}$ that is in a rectangle $r_j$ by $o_j$.
Graphically, assuming
the first case map, when we draw the rectangle decomposition
corresponding to the standard graph as in figure \ref{pair} we get Figure \ref{rec}. 
\begin{figure}[h]\centering
\includegraphics[width=7cm]{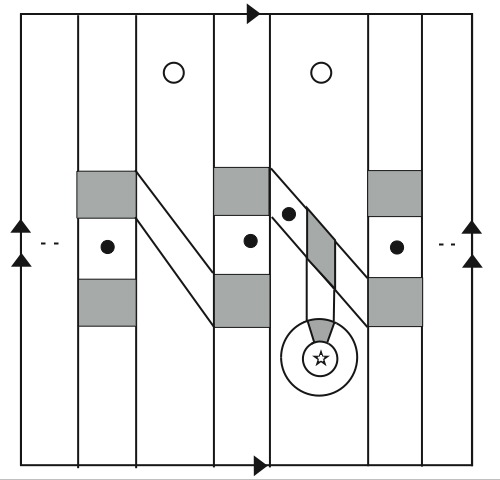}
\caption{The intermediate orbit denoted by black circles. The gray areas are junction, which can be deleted, and the rectangles can be glued directly to one another.}\label{rec}
\end{figure}

We will now show  that for any simple pair $\wh{\frac{q_1}{p_1}}\vee\wh{\frac{q_2}{p_2}}$ the
orbit $O_{1,1}$ is a simple orbit, and forms a simple pair with each periodic orbit of the pair, that is with $\widehat{\frac{q_1}{p_1}}$ and $\widehat{\frac{q_2}{p_2}}$ .
For the first assertion, we define a family of vertical loops as
follows: we choose a vertical loop that crosses both rectangles
corresponding to the $m$ line and passes to the right of the periodic point
$o_{p_1+4p_2+1}$. Denote this loop by $A$. It is shown graphically in figure \ref{new_orbit}. All
its images under $f$ until the $p_1$st iteration
are exactly of the same form, as the rectangles are
simply mapped to the right without changing their forms. Its
$p_1$st image is the first time it returns to the same rectangles, and is determined by the images of the corresponding vertical edges of the graph. These images are shown in figure \ref{pair_im}. 
We use the fact $f$ preserves orientation
to determine the relation between the image of the curve and the points of $O_{1,1}$. We denote this image by $B$. It is shown in figure \ref{new_orbit}. 

\begin{figure}[h]\centering
\includegraphics[width=7cm]{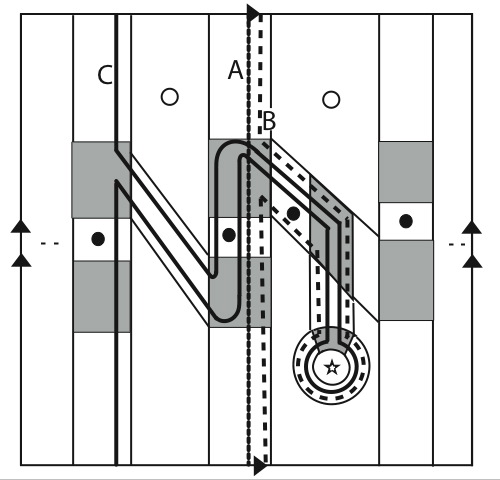}
\caption{The intermediate orbit with the family of vertical loops}
\label{new_orbit}
\end{figure}
By similar considerations, knowing
the rectangles containing $B$ in the original picture (Figures \ref{pair} and \ref{pair_im})
the rectangle adjacent
to $f^{p_1-1}(m)$ on the left contains a point of the $y$-orbit
and therefore contains a rectangle of type $L$ of Markov
partition. This rectangle contains the point $o_{3p_2}$ of the new
orbit. Line $A$ lies to the right of $m$ and of $o_{p_1+4p_2+1}$
therefore $f^{p_1-1}(A)$ lies to the right of $o_{3p_2}$ and to
the right of $o_{2p_1+4p_2}$. It follows that the line
$B=f^{p_1}(A)$ lies to the right of $o_{p_1+4p_2+1}$ and to the
right of $o_{2p_2+1}$, as shown on Figure \ref{new_orbit}. Also $B$ can be
isotoped to the right of $A$ relative to the points of the new
periodic orbit. Next $p_2-1$ iterations of $f$ translate $A$ and
$B$ and whole rectangle adjacent to $m$ on the right to the right.
We arrive at the rectangle adjacent to  $l$ on the left containing
point $o_{3p_2}$ of the new orbit. The point $o_{2p_1+4p_2}$ lies in a
rectangle on the line $l$. The loop $f^{p_2-1}(B)$ lies to the right of
$o_{3p_2}$ and to the left of $o_{2p_1+4p_2}$ therefore the loop
$C=f^{p_2}(B)$ lies to the right of $o_{p_1+4p_2+1}$ and to the
left of $o_{2p_2+1}$, as shown on Figure \ref{new_orbit}. Also $p_2$ iterations
of $f$ take line $m$ to a distance $p_2q_1=p_1q_2-1$ rectangles to
the right, which means one rectangle to the left of line $m$.
Since $A$ and $B$ are to the right of the point $o_{p_1+4p_2+1}$
in line $m$ and since this point moves to the leftmost point in
the new periodic orbit shown on Figure \ref{new_orbit}, loop $C$ must be to the
right of it, as in Figure \ref{new_orbit}. The point $o_{2p_1+1}$ may be above or
below the loop $C$ but this does not change the 
discussion bellow.

Note that if we disregard the orbits $x$ and $y$ of the original
pair we can isotop $C$ to $A$ relative the points of the new
orbit. This shows that the new orbit is a simple periodic orbit of
length $p_1+p_2$, by Lemma \ref{invariant loops}. Now we fill in the $x$ orbit (the
longer orbit) and consider a torus punctured at the $y$ orbit and
the new orbit together. We have the family of vertical loops
$f^i(A)$ and the action on it is exactly as in the condition for a
simple pair as the loop $C$ can be isotoped to $A$ plus a loop
around $y_0$. We choose a horizontal loop as the loop $D$ on
Figure \ref{new2}. Then its image $D'$ is as shown on Figure \ref{new2}. 
The image has the required properties. The orbit $y$
together with the new orbit form a simple pair for the
homeomorphism $f$.

\begin{figure}[h]\centering
\includegraphics[width=7cm]{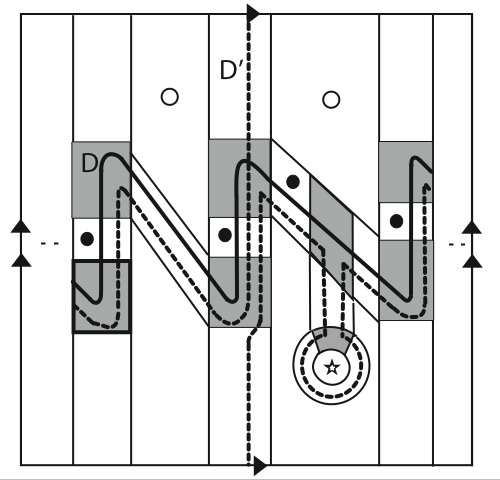}
\caption{}\label{new2}
\end{figure}

Next we fill in the $y$ orbit and leave punctures at the $x$ orbit
and the new orbit. We choose the initial vertical loop $A$
differently, as in Figure \ref{new and x}. This loop $A$ is one rectangle to the right
of $m$ plus a loop on the left. After $k=p_2-1$ iterations of $f$
it will move to line $l$ which is $q_1$ rectangles to the left of
$m$. Indeed it will move to $q_1(p_2-1)+1=p_1q_2-q_1$ rectangles
to the right of $m$ which means $q_1$ rectangles to the left. The
loop $f^k(A)$ looks like the loop $A$ and lies to the left of the
point $r_{3p_2}$ and to the left of the point $o_{2p_1+4p_2}$.
Next iteration of $f$ takes it to a curve which looks like $f(l)$
but lies to the left of $o_{p_1+4p_2+1}$ and to the left of
$o_{2p_2+1}$. Since we filled the point $y_0$ we can isotop this
loop to a vertical loop near $m$, which passes to the left of
$o_{p_1+4P_2+1}$. Subsequent iterations translate it to the right
and $f^{p_1}(A)$ is equal to curve $B$ on Figure \ref{new and x} 

\begin{figure}[h]\centering
\includegraphics[width=7cm]{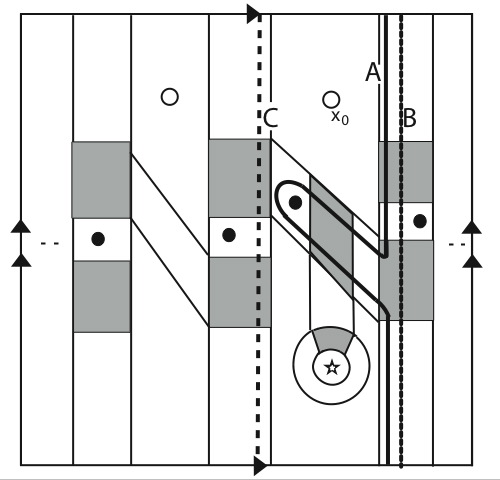}
\caption{}\label{new and x}
\end{figure}
Next $p_2-1$
iterations will take $B$ to the loop near $l$ which lies to the
left of $o_{2p_1+4p_2}$. Next iteration of $f$ takes this loop to
a loop similar to $f(l)$, but lies to the left of $o_{2p_2+1}$. Since
the points of $y$-orbit are filled we can isotop it to the loop
$C$ on Figure \ref{new and x}. It can be further isotoped, relative to the
$x$-orbit and the new orbit, to the loop $A$ plus a small loop
around $x_0$ to the left of $A$. If we choose the same horizontal
loop as in the previous case , with the same image as before, we
get the required action of $f$ for a simple pair consisting of the
$x$-orbit and the new orbit.

Now we can
continue by the same analysis for each of these two simple pairs, finding
their Farey intermediate to be a simple orbit as well
that forms a simple pair with each of them, and so on. It remains to prove the persistence of all these simple pairs under isotopies. For this, recall The following theorem from  \cite{Hall}.

\begin{thm}{\bf(Hall)} \label{hall_thm} Let $S$ be a closed surface and let $A$ be a finite subset
of $S$.
   Let $f$ be a homeomorphism of $S$ which leaves $A$ invariant. Let ${\bf
p}=x_1,\dots,x_k$ be a finite collection of periodic points for
$f$ which are essential, uncollapsible, mutually non-equivalent
and non-equivalent to points of $A$. Then the collection $\bf p$
is unremovable, which means that for every homeomorphism $g$
isotopic to $f$ rel $A$ there exists an isotopy $f_t$ rel $A$ and
paths $x_i(t)$ in $S$ such that $f_0=f$, $f_1=g$, $x_i(0)=x_i$,
$x_i(t)$ is a periodic point of $f_t$ of period equal exactly to
the period of $x_i$.\end{thm}

(This theorem is a generalization of the main result of Asimov and Franks
in \cite{AF} to several periodic orbits. In fact this generalization was
mentioned in \cite{AF} as a remark with a hint of a proof.)

Recall also that  if $f$ is pseudo-Anosov in the complement of $A$
then it is condensed and by \cite{Boy1} Lemma 1 and Theorem 2.4 each periodic point is
uncollapsible and essential and points from different orbits are
non-equivalent and points disjoint from $A$ are not equivalent to
points of $A$.

\begin{cor}: Let $T$ be a torus and let $A$ be a finite subset of \  $T$. Let
$f$ be a shear-type homeomorphism of $T$ which is pseudo-Anosov in
the complement of $A$. Let g be a homeomorphism of $T$ isotopic to
$f$ in the complement of $A$. If ${x}$ is a simple periodic orbit
for $f$ then there exists a simple periodic orbit ${z}$ for $g$
with $\rho(z)=\rho(x)$. If $ x,y $ is a simple pair of periodic
orbits for $f$, one or both disjoint from $A$, then there exists a
simple pair of periodic orbits $ z,w$ for g with $\rho(z)=\rho(x)$
and $\rho(w)=\rho(y)$.\end{cor}

Proof. Chose points $x_1$ and $y_1$ from the orbits $x$ and $y$ .
By Theorem \ref{hall_thm} there exists an isotopy $f_t$ and paths $x_1(t)$ and
$y_1(t)$ such that $x_1(0)=x_1$,  $y_1(0)=y_1$, $x_1(t)$ is a
periodic point of $f_t$ of a fixed order $p$ for all $t$  and
$y_1(t)$ is a periodic point of $f_t$ of a fixed order $q$ for all
$t$ and $y_1(t)=y_1(0)$ for all $t$ if $y(0)\in A$. For a given
$t$ all points in the orbits of $ x_1(t)$ and $ y_1(t)$  for $f$
are distinct, they form a braid with $p+q$ strands. They move when
$t$ changes and their movement can be extended to an ambient
isotopy $h_t$ which is fixed on $A$. Then
$h_t(f^i(x_1(0))=f_t^i(x_1(0))$ and
$h_t(f^i(y_1(0))=f_t^i(y_1(0))$. Consider isotopy
$F_t=h_t^{-1}f_th_t$. We have $F_t(f^i(x_1(0))=f^{i+1}(x_1(0))$
and $F_t(f^i(y_1(0))=f^{i+1}(y_1(0))$ so $F_t$ is fixed on the
orbits $x$ and $y$. In particular $x$ and $y$ form a simple pair
of periodic orbits for $F_1$ (or $x$ forms a simple periodic orbit
for $F_1$ if there is no $y$). But $F_1=h_1^{-1}gh_1$ so $h_1(x)$
and $h_1(y)$  form a simple pair of periodic orbits for $g$.

This concludes the proof of Theorem \ref{mainthm}.
\section{Global analysis of the kicked accelerated particle system \label{global}}

The physical system called the kicked accelerated particle consists of 
particles that do not interact with one another. They are subject to gravitation and so fall downwards, and are kicked by an
electro-magnetic field, i.e., the electro magnetic field is turned on for a very short time once in a fixed time interval. This electromagnetic field is a sine function of the height of the particle, hence the particles are kicked upwards or downwards by different amounts, depending on their position at the time of a kick.
For a short review of the results
for this system see \cite{review}. Experiments of this system were conducted
by the Oxford group, see \cite{ox}, and the system was found to
show a phenomena that is now called "quantum accelerator modes":
as opposed to the natural expectation that particles fall with
more or less the gravitational acceleration, it was found that a
finite fraction of the particles fall with constant nonzero
acceleration relative to gravity, as can be seen in Figure \ref{exp}
\begin{figure}[h]\centering
\includegraphics[width=13cm]{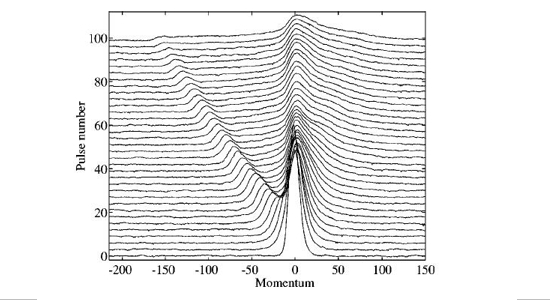}
\caption{Accelerator modes} Experimental Data (taken from Oberthaler, Godun, d'Arcy, Summy and Burnett, see \cite{ox}) showing the
number of atoms with specified momentum relative to the free
falling frame as the system develops in time (the numbers on the
$y$ axis represents time by the number of kicks, while the $z$ coordinate is
proportional to the number of atoms)\label{exp}
\end{figure}\\
This is a truly quantum phenomenon having no counterpart in the
classical dynamics. A theoretical explanation for this phenomenon
was given by Fishman, Guanieri and Rebuzzini in \cite{FGR}, and it
establishes a correspondence between accelerator modes of the
physical system, and periodic orbits of the classical map

\bc \bea
f:\left(%
\begin{array}{c}
  J \\
  \theta \\
\end{array}%
\right)\mapsto \left(%
\begin{array}{c}
  J+\tilde{k}sin(\theta+J)+\Omega \\
  \theta+J \\
\end{array}%
\right)\ mod2\pi \label{f} \eea \ec

Where the $J$ coordinate corresponds to the particles momentum, and $\theta$ to its coordinate. This  map is of shear type, and the acceleration for a periodic
orbit with rotation number $\frac{q}{p}$ is given by
\bc \bea
\alpha=\frac{2\pi q}{p}-\Omega \label{acceleration value}
\eea \ec

Hence, by analyzing the structure of existence of periodic orbits
for the classical map above, we would be able to find which modes
should be expected for which values of the parameters $k$ and
$\Omega$. We remark that actual experimental observation also
requires stability of the periodic orbits. It is important to
stress here that since these parameters correspond to the kick
strength and the time interval between kicks they can be
controlled in the experiments as we wish, so results obtained
for this system can be tested experimentally. 
When one plots the numerical results describing which periods
exist for different values of $k$ and $\Omega$ one gets an
extremely complicated figure, see figure \ref{tongues}.
\begin{figure}[h]\centering
\includegraphics[width=13cm]{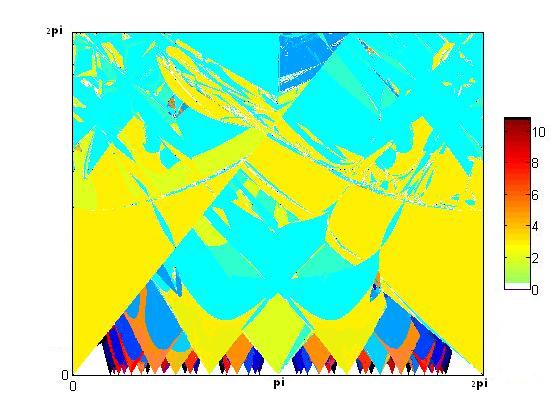}
\caption{Tongues of periodic orbits \label{tongues}}
\end{figure}

An exact mathematical analysis of this system is extremely
complicated. Perturbative methods have been used in \cite{FGR} to
analyze the existence of these "tongues" of periodic orbits in the
region where $k\rightarrow 0$, as well as giving estimates on
their widths. 

Look at the map $f$ given by (\ref{f}) in regions
where $\Omega$ is equal $\frac{q}{p} 2\pi$ for some rational
number $\frac{q}{p}$ in the unit interval, and small $k$. For a
small enough $k$ it can be seen both from the numerical results
shown graphically in figure \ref{tongues} and from perturbative
arguments that in the above region a periodic orbit with period
$p$ exists. 

For small $k$ the periodic points of this orbit 
must be pretty much equally spaced along the $J$ axis, and we can
choose (for $k$ small enough) a family of vertical loops
that are
equally spaced at distance exactly $\Omega$ apart, and each is at
distance at least, say, $3k$ from any of the periodic points.

The image of a
loop parameterized by $\Gamma_1(\theta)=\left(%
\begin{array}{c}
  J_0 \\
  \theta \\
\end{array}%
\right)$ is given by\\
$\left(%
\begin{array}{c}
  J_0+k sin(J_0+\theta)+\Omega \\
  J_0+\theta \\
\end{array}%
\right)=\left(%
\begin{array}{c}
  J_0+k sin(\theta') +\Omega \\
  \theta' \\
\end{array}%
\right)$ and so is very close (for small $k$) to another loop of the chosen family. It follows that there exists a map $\tf$ isotopic
to $f$ rel the orbit which keeps this family of curves invariant, and so, by
Lemma \ref{invariant loops}, all the periodic orbits seen in the
tips of the tongues in Figure \ref{tongues} are simple orbits.

Note the rotation number of each of
these orbits is equal exactly to the value of $\Omega$ in the tip
of the tongue ($k=0$) as for very small $k$ the $J$ coordinate
increases by an almost fixed value, close as we wish to $\Omega$.
And, by equation (\ref{acceleration value}) the rotation number $\frac{q}{p}$ is related to the acceleration of the
corresponding acceleration mode by

\bc $\alpha=\frac{q}{p}2\pi-\Omega$ \ec

So the topological meaningful numbers here are in fact also the
ones with physical significance. While $\Omega$
changes through the region in which this periodic orbit exists,
$\frac{q}{p}$ is of course a topological invariant and
therefore fixed. Hence the acceleration
vanishes on the line with fixed $\Omega$ in the middle of each tongue, and changes signs when one
crosses this line. This was measured experimentally in \cite{exp}.

For any other point higher in the tongue which we can reach by
an isotopy along which the periodic orbit exists, we also have the
orbit is a simple orbit. We will assume, as is very natural and was checked 
numerically for many cases, that the orbits remain simple throughout
the region of each tongue.

In some of the cases for which we drew a portrait of
the phase space, we found that the fact the homeomorphism is
isotopic to one which is reducible rel the periodic orbit is
realized by the physical map itself, as seen in Figure \ref{2}.
\begin{figure}[h]\centering
\includegraphics[width=11cm]{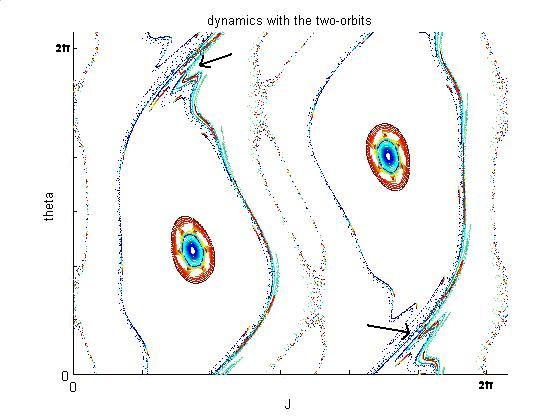}
\caption{Phase portrait for a two-orbit}\label{2} Drawn for
$k=\frac{1}{15}2\pi$ and $\Omega=\pi$, the two-orbit which is
clearly seen is a stable orbit with two stable neighborhoods
drawn. There is another two-orbit present, at which the arrows
point, and it is the stable and unstable manifolds for this
unstable orbit which divide the phase space into non intersecting
regions which do not mix.
\end{figure}

Here the phase space is truly divided into pieces. Each of the annuli in this decomposition is
mapped to another, and returns to itself with one twist
after $p$ iterations of $f$. Therefore every periodic orbit must
have a period which is a multiple of $p$. On the other hand, when an annulus is mapped to itself with
one twist under an area preserving map (here under $f^p$), every rotation number in the unit
interval exists for it (here we mean the standard annulus rotation number measuring the rotations around the annulus), and so every
period exists, as for every rational number $\frac{n}{m}$ there is
a periodic point of order $m$ which rotates $n$ times around the
annulus before it returns to itself. This yields that for such a point in the parameter space, exactly all periods that are multiples of $p$ exist. 
It is our belief that this situation is typical for the
center of each tongue, that is for $\Omega=2\pi\frac{q}{p}$. At
other points, namely in all point we have numerically checked
outside the center of the tongue, orbits of coprime lengths may
exist simultaneously.
We believe that the coexisting orbits
whose rotation numbers are Farey neighbors form a simple pair
together, as in the example on Figure \ref{2and3}, which shows a simple
pair of orbits with rotation numbers $1/3$ and $1/2$ found in the
physical system.
\begin{figure}[h]\centering
\includegraphics[width=11cm]{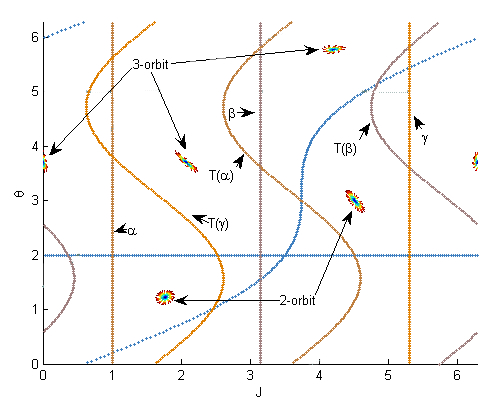}
\caption{A pair of coexisting orbits in the physical system}\label{2and3} Drawn with a collection
of curves on the torus and their images, which show this is a simple pair.
\end{figure}

This coexistence happens at a point in Figure \ref{tongues} for which two
tongues intersect. We assume
that the same orbit persists throughout the tongue, and therefore we have at such a point
two coexisting simple orbits.
We believe
that in all points of intersecting tongues coming from $k=0$ and
$\Omega_1=\frac{q_1}{p_1}$, $\Omega_2=\frac{q_2}{p_2}$ which are
Farey neighbors, $p_1>p_2$, the coexisting orbits form a simple
pair. 

Theorem \ref{con:infty} therefore implies that there are
infinitely many periodic orbits for the parameters at a region of
intersection of two such tongues, with rotation numbers equal to all rational numbers between the ones of these
two tongues. If we assume all these simple orbits present also
come from tongues,  this yields that each rational tongue between
$\frac{q_1}{p_1}$ and $\frac{q_2}{p_2}$ intersects each of
these two tongues lower (along the $k$ axis) than they intersect
each other. In other words, following a path from a tip of a tongue upwards in the tongue, if it intersects a Farey neighbor tongue we know it intersects earlier all tongues of rational numbers between them.
This
determines the global structure appearing in
Figure \ref{tongues} of all accelerator modes in the physical system, as Sharkovskii's theorem determines it for one dimensional systems.

\vskip 1cm

Tali Pinsky

Department of Mathematics

The Technion

32000 Haifa, Israel

e-mail:  otali@tx.technion.ac.il

\par\bigskip

Bronislaw Wajnryb

Department of Mathematics 

Rzeszow University of Technology

ul. W. Pola 2, 35-959 Rzeszow, Poland

e-mail:  dwajnryb@prz.edu.pl
\par\bigskip

\end{document}